\documentclass[english]{article}
\usepackage[T1]{fontenc}
\usepackage[latin9]{inputenc}
\usepackage{refstyle}
\usepackage{dvipost}
\usepackage{mathrsfs}
\usepackage{enumitem}
\usepackage{amsmath}
\usepackage{amsthm}
\usepackage{amssymb}
\usepackage{setspace}
\usepackage{cleveref}
\usepackage{ytableau}
\usepackage{mathdots}
\usepackage{graphicx}
\usepackage{wasysym}
\makeatletter

\AtBeginDocument{\providecommand\thmref[1]{\ref{thm:#1}}}
\AtBeginDocument{\providecommand\caseref[1]{\ref{case:#1}}}
\AtBeginDocument{\providecommand\lemref[1]{\ref{lem:#1}}}
\RS@ifundefined{subref}
 {\def\RSsubtxt{section~}\newref{sub}{name = \RSsubtxt}}
 {}
\RS@ifundefined{thmref}
 {\def\RSthmtxt{theorem~}\newref{thm}{name = \RSthmtxt}}
{}
\RS@ifundefined{lemref}
 {\def\RSlemtxt{lemma~}\newref{lem}{name = \RSlemtxt}}
{}

\dvipostlayout
\dvipost{osstart color push Red}
\dvipost{osend color pop}
\dvipost{cbstart color push Blue}
\dvipost{cbend color pop}

\numberwithin{figure}{section}
\numberwithin{equation}{section}
\theoremstyle{plain}
\newtheorem{thm}{\protect\theoremname}[section]
  \theoremstyle{plain}
  \newtheorem{cor}[thm]{\protect\corollaryname}
  \theoremstyle{plain}
  \newtheorem{prop}[thm]{\protect\propositionname}
 \newlist{casenv}{enumerate}{4}
 \setlist[casenv]{leftmargin=*,align=left,widest={iiii}}
 \setlist[casenv,1]{label={{\itshape\ \casename} \arabic*.},ref=\arabic*}
 \setlist[casenv,2]{label={{\itshape\ \casename} \roman*.},ref=\roman*}
 \setlist[casenv,3]{label={{\itshape\ \casename\ \alph*.}},ref=\alph*}
 \setlist[casenv,4]{label={{\itshape\ \casename} \arabic*.},ref=\arabic*}
  \theoremstyle{plain}
  \newtheorem{lem}[thm]{\protect\lemmaname}
  \theoremstyle{definition}
  \newtheorem{example}[thm]{\protect\examplename}
  \theoremstyle{remark}
  \newtheorem{rem}[thm]{\protect\remarkname}

\AtBeginDocument{\providecommand\caseref[1]{\ref{case:#1}}}
\RS@ifundefined{caseref}
  {\newref{case}{name = case~,names = cases~}}
  {}

\date{}

\usepackage{tikz} 
\usepackage{tocbibind}
\usepackage{youngtab}
\usepackage{authblk}

\DeclareMathOperator{\Hom}{Hom}
\DeclareMathOperator{\End}{End}

\DeclareMathOperator{\tr}{tr}
\DeclareMathOperator{\trace}{trace}
\DeclareMathOperator{\Rad}{Rad}
\DeclareMathOperator{\im}{\mathsf{im}}

\DeclareMathOperator{\Span}{span}

\DeclareMathOperator{\PT}{PT}
\DeclareMathOperator{\PO}{PO}
\DeclareMathOperator{\F}{F}
\DeclareMathOperator{\PF}{PF}
\DeclareMathOperator{\PC}{PC}
\DeclareMathOperator{\IS}{IS}
\DeclareMathOperator{\T}{T}

\DeclareMathOperator{\Lc}{\mathscr{L}}

\DeclareMathOperator{\Jc}{\mathscr{J}}

\DeclareMathOperator{\rank}{rank}
\DeclareMathOperator{\dom}{\mathsf{dom}}
\DeclareMathOperator{\Irr}{\mathsf{Irr}}
\DeclareMathOperator{\IRR}{\mathsf{IRR}}

\DeclareMathOperator{\Stab}{Stab}
\DeclareMathOperator{\Ind}{Ind}
\DeclareMathOperator{\Res}{Res}
\DeclareMathOperator{\Ext}{Ext}

\makeatother

\usepackage{babel}
  \providecommand{\corollaryname}{Corollary}
  \providecommand{\examplename}{Example}
  \providecommand{\lemmaname}{Lemma}
  \providecommand{\propositionname}{Proposition}
  \providecommand{\remarkname}{Remark}
 \providecommand{\casename}{Case}
\providecommand{\theoremname}{Theorem}

\begin{document}

\title{The representation theory of the monoid of all partial functions
on a set and related monoids as EI-category algebras}

\author{Itamar Stein\thanks{This paper is part of the author's PHD thesis, being carried out under the supervision of Prof. Stuart Margolis.\\The author's research was supported by Grant No. 2012080 from the United States-Israel Binational Science Foundation (BSF).}}

\affil{ Department of Mathematics\\
    Bar Ilan University\\
    52900 Ramat Gan\\
    Israel}

\maketitle

\begin{abstract}
The (ordinary) quiver of an algebra $A$ is a graph that contains
information about the algebra's representations. We give a description
of the quiver of $\mathbb{C}\PT_{n}$, the algebra of the monoid of
all partial functions on $n$ elements. Our description uses an isomorphism
between $\mathbb{C}\PT_{n}$ and the algebra of the epimorphism category,
$E_{n}$, whose objects are the subsets of $\{1,\ldots, n\}$ and morphism
are all total epimorphisms. This is an extension of a well known isomorphism
of the algebra of $\IS_{n}$ (the monoid of all partial injective maps
on $n$ elements) and the algebra of the groupoid of all bijections
between subsets of an $n$-element set. The quiver of the category algebra is described
using results of Margolis, Steinberg and Li on the quiver of EI-categories. We use the same technique to compute the quiver of other
natural transformation monoids. We also show that the algebra $\mathbb{C}\PT_{n}$ has three blocks
for $n>1$ and we give a natural description of the descending Loewy
series of $\mathbb{C}\PT_{n}$ in the category form. 
\end{abstract}
\Yvcentermath1


\section{Introduction}

One of the goals of the study of monoid representations is to relate
them to the modern representation theory of associative algebras.
Given a (finite) monoid $M$, it is of interest to study the properties
of its algebra $\mathbb{C}M$ (all representations in this paper are
over the field of complex numbers, $\mathbb{C}$). For instance, monoids
for which $\mathbb{C}M$ is semisimple are characterized in \cite[Chapter 5]{Clifford1961}
and monoids for which $\mathbb{C}M$ is basic are characterized in
\cite{Almeida2009} (along with some other natural types of algebras).
Another important invariant of an associative algebra is its ordinary
quiver. Saliola \cite{Saliola2007} described the quiver of left regular
bands and Denton, Hivert, Schilling and Thi{\'e}ry \cite{Denton2010} described the quiver of $\Jc$-trivial monoids. These results were generalized by Margolis and Steinberg in \cite{Margolis2012}
where they described the quiver of a class of monoids called rectangular
monoids that includes properly all the finite monoids whose algebras
are basic. By a description, we mean that they reduced the computation
of the quiver to a problem in the representation theory of the maximal
subgroups of the monoid. They also characterized \cite{Margolis2011}
regular monoids whose algebras are directed or co-directed with respect to their natural
quasihereditary structure. That means that the quiver has only arrows going upwards or downwards with
respect to the partial order on vertices induced from the $\Jc$-order of the monoid. Apart from that, not much
is currently known about quivers of monoid algebras. The case of classical
transformation monoids is clearly of interest. Since the algebra of
$\IS_{n}$ (the symmetric inverse monoid) is semisimple, its quiver
has no arrows. However, for the monoids of partial and total functions
on $n$ elements, denoted $\PT_{n}$ and $\T_{n}$, we get an interesting
question. Putcha \cite{Putcha1996} observed that $\mathbb{C}\PT_{n}$
is co-directed (if we multiply from right to left), that means that
all the arrows in the quiver are goings downwards. He \cite{Putcha1998}
also computed the quiver of $\mathbb{C}\T_{n}$ up to $n=4$ and partially computed
the quiver for $n>4$. This was enough to deduce that $\mathbb{C}\T_{n}$
has two blocks for $n>3$ and that it is not of finite representation
type for $n>4$. We remark that Ponizovski \cite{Ponizovski1987}
proved that $\mathbb{C}\T_{n}$ is of finite representation type for
$n\leq3$ and Ringel \cite{Ringel2000} proved the same for $n=4$
along with finding relations for the quiver presentation. Recently,
Steinberg \cite{Steinberg2015} showed that the quiver of $\mathbb{C}\T_{n}$
is acyclic and that the global dimension of $\mathbb{C}\T_{n}$ is
$n-1$.

\paragraph{}

In this paper we give a full description of the quiver of $\mathbb{C}\PT_{n}$.
It is known \cite{Steinberg2006} that there is an isomorphism between
$\mathbb{C}\IS_{n}$ and the groupoid algebra of the groupoid of bijections
on an $n$-element set. We extend this isomorphism to an isomorphism
between $\mathbb{C}\PT_{n}$ and the algebra of the epimorphism category
$E_{n}$, that is, objects are subsets of $\{1,\ldots, n\}$ and the
morphisms are all total epimorphisms. The infinite version of this category was also studied
in \cite[Chapter 8]{Sam2014}. $E_{n}$ is not a groupoid, but is
an EI-category, that is, any endomorphism is an isomorphism. A
formula for computing the quiver of skeletal EI-categories was
found independently by Margolis and Steinberg \cite[Theorem 6.13]{Margolis2012}
and Li \cite{Li2011}. We use their formula to get a description of
the quiver by means of representations of the symmetric group. We
then use standard tools from the theory of representations of the
symmetric group to get a combinatorial description of the number of
arrows between any two vertices. We also deduce that the algebra $\mathbb{C}\PT_{n}$
has three blocks (for $n>1)$ and we give a natural description for
the Loewy series of $\mathbb{C}\PT_{n}$ in the category form. We
then compute the quiver of the algebras  of other natural transformation monoids related to
$\PT_{n}$: The monoid of all order-preserving
partial functions, the monoid of all order-decreasing partial functions, the monoid of all order-decreasing total
functions, and the partial Catalan monoid which is the intersection of the first two.

\section{Preliminaries}

\subsection{Representations}

Recall that an \emph{algebra} over a field $K$ is a ring $A$ that
is also a vector space over $K$ such that $k(ab)=(ka)b=a(kb)$ for
all $k\in K$ and $a,b\in A$. We will consider only unital, finite dimensional
$\mathbb{C}$-algebras, that is, $K$ will always be the field of
complex numbers and $A$ will be a unital ring that has finite dimension as a $\mathbb{C}$-vector
space. A \emph{representation} of $A$ is a $\mathbb{C}$-algebra
homomorphism $\rho:A\rightarrow\End_{\mathbb{C}}(V)$ where $V$ is
some (finite dimensional) vector space over $\mathbb{C}$. Equivalently we
can say that a representation of $A$ is a (finitely generated, left)
module over $A$. A non-zero representation $M$ is called \emph{irreducible}
or \emph{simple} if it does not have proper submodules other than
$0$. The set of all irreducible representations of $A$ is denoted
$\Irr A$. A representation $M$ is called \emph{semisimple} if it
is a direct sum of simple modules. The algebra $A$ is called \emph{semisimple}
if any $A$-module is semisimple. We denote by $\Rad M$ the \emph{radical} of a module $M$, which is the minimal submodule such that $M/\Rad M$ is a semisimple module. The radical of $A$ is its radical as a left module over itself, which is also the minimal ideal such that $A/\Rad A$ is a semisimple algebra. It is
well known that $\Rad A$ is also the maximal nilpotent ideal of $A$
and the intersection of all maximal ideals of $A$. Clearly $A$ is
semisimple if and only if $\Rad A=0$. The \emph{descending Loewy
series} of a module $M$ is the decreasing sequence of submodules
\[
0\subsetneq\ldots\subsetneq\Rad^{2}M\subsetneq\Rad M\subsetneq M
\]
and the minimal integer $n$ such that $\Rad^n M=0$ is called the \emph{Loewy length} of $M$.

Recall that a non-zero idempotent $e\in A$ is called \emph{primitive} if
the existence of idempotents $f,f^{\prime}\in A$ such that $f+f^{\prime}=e$
and $ff^{\prime}=f^{\prime}f=0$ implies that $f=e$ and $f^{\prime}=0$
(or vice versa). It is known that any irreducible module $N$ is isomorphic
to $Ae/\Rad(Ae)$ for some primitive idempotent $e$.

Recall that the \emph{ordinary quiver} $Q$ of a finite dimensional
algebra $A$ is a directed graph defined in the following
way: The vertices of $Q$ are in a one-to-one correspondence with
the irreducible representations of $A$ (up to isomorphism). If $N_{i}$
and $N_{j}$ are irreducible representations of $A$ (identified with
two vertices of the quiver) then the number of arrows from $N_{i}$
to $N_{j}$ is

\[
\dim\Ext^{1}(N_{i},N_{j})
\]

which is known to be equal to

\[
\dim e_{j}(\Rad A/\Rad^{2}A)e_{i}
\]
where $e_{i},e_{j}$ are primitive idempotents corresponding to $N_{i}$
and $N_{j}$ (this number is independent of the specific choice of
idempotents). Note that an algebra $A$ is semisimple if and only
if its quiver has no arrows at all. More about representations of algebras
and quivers can be found in \cite{Assem2006}.

\subsection{Monoids and monoid representations}

Throughout this paper, all monoids are assumed to be finite. Two elements
$a,b$ of a finite monoid $M$ are \emph{$\mathscr{J}$-equivalent}
if they generate the same principal ideal, that is
\[
a\Jc b\Leftrightarrow MaM=MbM.
\]

A group $H$ which is a subsemigroup of $M$ is called a \emph{subgroup}
of $M$ (but $M$ and $H$ do not necessarily have the same unit element).
A $\Jc$-class is called \emph{regular} if it contains an idempotent.
It is clear that any subgroup $H\subseteq M$ is contained in some
regular $\Jc$-class. It is well known that in a finite monoid any
two maximal subgroups in the same $\Jc$-class are isomorphic. We
denote by $\T_{n}$ and $\PT_{n}$ the monoids of all total and partial
functions on $n$ elements, respectively. $\IS_{n}$ denotes the monoid
of all injective partial functions on $n$ elements. These monoids are fundamental
in monoid theory. For instance, \cite{Ganyushkin2009b} is solely
devoted to their study. Note that $\PT_{n}$ and $\IS_{n}$ are partially
ordered by containment of relations. For other basics of semigroup
theory the reader is referred to \cite{Howie1995}.

We denote by $\mathbb{C}M$ the \emph{monoid algebra} of $M$ over $\mathbb{C}$,
which is the $\mathbb{C}$-vector space of formal sums $\{\sum\alpha_{i}m_{i}\mid m_{i}\in M\}$
where the multiplication is induced by the multiplication of $M$.
Since $M$ is a finite monoid, $\mathbb{C}M$ is an associative, unital
and finite dimensional algebra. Clearly $\dim\mathbb{C}M=|M|$. A $\mathbb{C}M$-module is also called an $M$-module. We also denote the set of irreducible $M$-representations by $\Irr M$.

The case where $M=G$ is a group is of special importance. Maschke's
theorem says that $\mathbb{C}G$ is always a semisimple algebra and
it is known that the number of irreducible $\mathbb{C}G$-modules
equals the number of conjugacy classes of $G$. We denote the trivial
representation of any group $G$ by $\tr_{G}$. Recall that if $V$
is a $G$-representations, then $V^{\ast}=\Hom(V,\mathbb{C})$ is
also a $G$-representation with operation $(g\cdot\varphi)(v)=\varphi(g^{-1}v)$.
The \emph{character} of a representation $\rho:\mathbb{C}G\to\End V$
is the function $\chi_{\rho}:G\to\mathbb{C}$ defined by $\chi_{\rho}(g)=\trace\rho(g)$.
It is well known that two $G$-representations $\rho$ and $\psi$
are isomorphic if and only if $\chi_{\rho}=\chi_{\psi}.$ Moreover,
if $S_{i}$, $1\leq i\leq r$ are the irreducible modules of $\mathbb{C}G$
with characters $\chi_{i}$, $1\leq i\leq r$, and the decomposition
of a $\mathbb{C}G$-module $N$ is
\[
N=\bigoplus_{i=1}^{r}a_{i}S_{i}
\]
then 
\[
a_{i}=\langle\chi_{N},\chi_{i}\rangle=\frac{1}{|G|}\sum_{g\in G}\chi_{N}(g)\overline{\chi_{i}(g)}.
\]

In order to simplify notation, we sometimes omit the $\chi$ and write
$N$ also for the character of $N$.

Let $H\subseteq G$ be a subgroup of $G$ and let $N$ and $L$ be
modules of $G$ and $H$, respectively. We denote by $\Res_{H}^{G}N$
and $\Ind_{H}^{G}L$ the \emph{restriction} and \emph{induction} modules.
We also recall the Frobenius reciprocity theorem which states that
\[
\langle\Ind_{H}^{G}L,N\rangle=\langle N,\Res_{H}^{G}L\rangle
\]
where here $M$, $N$, $\Ind_{H}^{G}L$ and $\Res_{H}^{G}L$ are the respective
characters. Recall that in the special case where $G=S_{n}$ is the
symmetric group, the irreducible representations correspond to partitions
of $n$ (or equivalently, to Young diagrams with $n$ boxes). We use
the standard notation for partitions, that is, a partition $\alpha_{1}+\ldots+\alpha_{k}=n$
is denoted $[\alpha_{1},\ldots,\alpha_{k}]$ where $\alpha_{i}\geq\alpha_{i+1}$.
Recall that the partition $[n]$ corresponds to the trivial representation
and $[1^{n}]=[1,\ldots,1]$ corresponds to the sign representation.
More about the basics of group representations can be found in \cite[Chapters 5-6]{Alperin1995}
and \cite{Sagan2001}.

Returning to the general case of monoid representations. We recall
the fundamental theorem of Munn-Ponizovski (see \cite[Thm 7]{Ganyushkin2009a}
for a modern proof).
\begin{thm}[Munn-Ponizovski]
\label{thm:MunnPonizovski} Let $M$ be a monoid and let $H_{1},\cdots,H_{n}$
be representatives of its maximal subgroups, one for every regular
$\Jc$-class. There is a one-to-one correspondence between irreducible
representations of $M$ and the irreducible representations of $H_{1},\cdots,H_{n}$.
\[
\Irr\mathbb{C}M\leftrightarrow{\displaystyle \bigsqcup_{k=1}^{n}}\Irr\mathbb{C}H_{k}.
\]

\end{thm}
As a result, we get a partial ordering of the irreducible modules.
Let $N_{1}$ and $N_{2}$ be irreducible $M$-modules, which correspond
to the $H_{i_{1}}$ and $H_{i_{2}}$-modules $V_{1}$ and $V_{2}$,
respectively. Then $N_{1}\leq N_{2}$ if $H_{i_{1}}\leq_{\Jc}H_{i_{2}}$.
Now, let $Q$ be the quiver of $\mathbb{C}M$. Since the vertices
of the quiver are in a one-to-one correspondence with the irreducible
representations of $M$ we get a partial ordering of the vertices
of $Q$ as well.

\subsection{Categories}

All categories in this paper are finite. Hence we can regard a category
$A$ as a set of objects denoted $A^{0}$, and a set of morphisms
denoted $A^{1}$. If $a,b\in A^{0}$ then $A(a,b)$ is the set of
morphisms from $a$ to $b$. A category is called a \emph{groupoid}
if any morphism is an isomorphism and an \emph{EI-category} if every
endomorphism is an isomorphism. If $A$ is a category, we can define
the \emph{category algebra} $\mathbb{C}A$ consisting of all linear
combinations of morphisms with obvious addition and multiplication.
Recall that if $f\in A(c,d)$ and $g\in A(a,b)$ are morphisms such
that $b\neq c$ then $fg=0$ in the category algebra. It is well known
that groupoid algebras are semisimple \cite[Section 3]{Steinberg2006}.
In fact, it is not difficult to check that an EI-category algebra $\mathbb{C}A$
is semisimple if and only if $A$ is a groupoid.

A morphism $f\in A^{1}$ is called \emph{irreducible} if it is not
left or right invertible but whenever $f=gh$, either $g$ is left
invertible or $h$ is right invertible. The set of irreducible morphisms
from $a$ to $b$ is denoted $\IRR A(a,b)$. Note that if $A$ is
an EI-category, a morphism $f$ is left or right invertible if and
only if it is an isomorphism. Indeed, let $f\in A(a,b)$ be a morphism
such that $fg=1_{b}$ for some $g\in A(b,a)$. Then $gf\in A(a,a)$
is an isomorphism so $f$ is left invertible as well. Hence, $f$
is irreducible if it is not an isomorphism and whenever $f=gh$, either
$g$ or $h$ is an isomorphism. Recall that two categories $A$ and $B$
are \emph{equivalent} if there is a fully faithful and essentially
surjective functor $F:A\to B$. Note that any category $A$ is equivalent
to some skeletal category $B$ (where \emph{skeletal} means that no
two objects of $B$ are isomorphic). If $A$ and $B$ are equivalent
categories then their algebras are Morita equivalent \cite[Proposition 2.2]{Webb2007}.

\subsection{M\"{o}bius functions}

Let $(X,\leq)$ be a finite poset. We view $\leq$ as a set of ordered pairs. The \emph{M\"{o}bius function} of $\leq$
is a function $\mu:\leq\to\mathbb{C}$ that can be defined in the
following recursive way:
\[
\mu(x,x)=1
\]
\[
\mu(x,y)=-\sum_{x\leq z<y}\mu(x,z)
\]

\begin{thm}[M\"{o}bius inversion theorem]

Let $V$ be a $\mathbb{C}$-vector space and let $f,g:X\to V$ be
functions such that
\[
g(x)=\sum_{y\leq x}f(y)
\]
then
\[
f(x)=\sum_{y\leq x}\mu(y,x)g(y).
\]
\end{thm}

More on M\"{o}bius functions can be found in \cite[Chapter 3]{Stanley1997}.
Important applications of M\"{o}bius functions to the representation theory
of finite monoids can be found in \cite{Steinberg2006} \cite{Steinberg2008}.

\section{The quiver of $\mathbb{C}\PT_{n}$}

Let $\PT_{n}$ be the monoid of all partial functions on $n$
elements. The goal of this section is to describe the quiver of the
algebra of $\PT_{n}$. We remark that in this paper we compose functions
from right to left. It is well known that $t\Jc s$ in $\PT_{n}$
if and only if $\rank s=\rank t$, that is, $|\im t|=|\im s|$. Hence
the $\Jc$-classes of $\PT_{n}$ are linearly ordered by rank. It
is also known that all the $\Jc$-classes are regular and the maximal
subgroup of the $\Jc$-class of rank $k$ is $S_{k}$. By \thmref{MunnPonizovski},
there is a one-to-one correspondence between $\Irr\PT_{n}$ and ${\displaystyle \bigsqcup_{k=0}^{n}}\Irr S_{k}$.
Since the irreducible representations of $S_{k}$ correspond to Young
diagrams with $k$ boxes (or partitions of $k$) it follows that the
vertices of the quiver of $\mathbb{C}\PT_{n}$ correspond to the Young
diagrams with $k$ boxes for $0\leq k\leq n$. For instance, we can
identify the vertices of the quiver of $\PT_{3}$ with the following
diagrams:

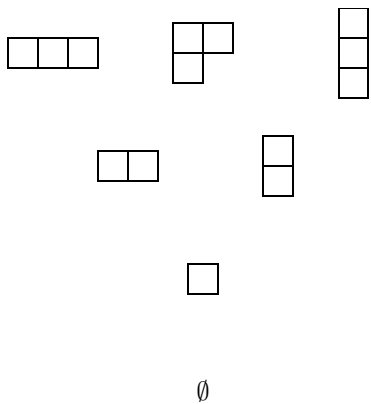
\begin{center} \begin{tikzpicture}\path (-2,3) node (S3_1) {$\yng(3)$}; \path (0,3) node (S3_2) {$\yng(2,1)$}; \path (2,3) node (S3_3) {$\yng(1,1,1)$}; \path (-1,1.5) node (S2_1) {$\yng(2)$}; \path (1,1.5) node (S2_2) {$\yng(1,1)$}; \path (0,0) node (S1_1) {$\yng(1)$}; \path (0,-1.5) node (S0_1) {$\emptyset$};\end{tikzpicture} \end{center}

We have ordered the diagrams according to the partial order on vertices
mentioned above. Hence, representations of $S_{r}$ appear above representations
of $S_{k}$ if $r>k$.

\paragraph{}

We will now show a way to describe the number of arrows between any
two Young diagrams in the quiver. We start by recalling a result from
\cite{Steinberg2006}. Denote by $G_{n}$ the category whose objects
are subsets of $\overline{n}=\{1,\ldots, n\}$ and morphisms are in
one-to-one correspondence with elements of $\IS_{n}$. For every $t\in\IS_{n}$
there is a morphism $G_{n}(t)$ from $\dom t$ to $\im t$, so multiplication
$G_{n}(s)G_{n}(t)$ is defined where $\im(t)=\dom(s)$ and the result
is $G_{n}(st)$. In other words, $G_{n}$ is the category of all bijections
between subsets of an $n$-element set. Note that $G_{n}$ is a groupoid. Note also
that restriction of functions (or containment of relations) is a partial
order on $\IS_{n}$ that turns $\IS_{n}$ into a partially ordered
monoid. We refer to this order as the natural order on $\IS_{n}$,
as it is a special case of the natural ordering of any inverse semigroup
(see \cite[Section 7.1]{Clifford1967} or \cite[Section 5.2]{Howie1995}).
The following theorem is \cite[Theorem 4.2]{Steinberg2006} for the
special case of the symmetric inverse monoid.
\begin{thm}
$\mathbb{C}\IS_{n}$ is isomorphic to $\mathbb{C}G_{n}$. Explicit
isomorphisms $\varphi:\mathbb{C}\IS_{n}\rightarrow\mathbb{C}G_{n}$,
$\psi:\mathbb{C}G_{n}\rightarrow\mathbb{C}\IS_{n}$ are defined (on
basis elements) by

\[
\varphi(s)=\sum_{t\leq s}G_{n}(t)
\]

\[
\psi(G_{n}(s))=\sum_{t\leq s}\mu(t,s)t
\]

where $\leq$ is the standard partial order on $\IS_{n}$ and $\mu$
is its M\"{o}bius function.
\end{thm}
We claim that this isomorphism can be extended into an isomorphism
between $\mathbb{C}\PT_{n}$ and the category of epimorphisms on an
$n$-element set, which we define now. Denote by $E_{n}$ the category whose
objects are subsets of $\overline{n}=\{1,\ldots, n\}$, and whose morphisms
are in one-to-one correspondence with the elements of $\PT_{n}$.
For every $t\in\PT_{n}$ there is a morphism $E_{n}(t)$ from $\dom t$
to $\im t$, so multiplication $E_{n}(s)E_{n}(t)$ is defined where
$\im(t)=\dom(s)$ and the result is $E_{n}(st)$. In other words,
$E_{n}$ is the category of all total onto functions between subsets
of an $n$-element set  (the ``$E$'' stands for epimorphisms). Note that $E_{n}$
is not a groupoid, but it is an EI-category since $E_{n}(X,X)\cong S_{|X|}$
where $X\subseteq\overline{n}$. Furthermore, the groupoid $G_{n}$
discussed above is precisely the groupoid of isomorphisms of the category
$E_{n}$.
\begin{prop}
\label{prop:MainIsomorphism}$\mathbb{C}\PT_{n}$ is isomorphic to
$\mathbb{C}E_{n}$. Explicit isomorphisms $\varphi:\mathbb{C}\PT_{n}\rightarrow\mathbb{C}E_{n}$,
$\psi:\mathbb{C}E_{n}\rightarrow\mathbb{C}\PT_{n}$ are defined (on
basis elements) by

\[
\varphi(s)=\sum_{t\leq s}E_{n}(t)
\]

\[
\psi(E_{n}(s))=\sum_{t\leq s}\mu(t,s)t
\]

where $\leq$ is the natural partial order on $\PT_{n}$ (containment
of relations) and $\mu$ is its M\"{o}bius function.\end{prop}
\begin{proof}
The proof that $\varphi$ and $\psi$ are bijectives is identical to what
is done in \cite{Steinberg2006}.
\begin{align*}
\psi(\varphi(s)) & =\psi(\sum_{t\leq s}E_{n}(t))=\sum_{t\leq s}\psi(E_{n}(t))\\
 & =\sum_{t\leq s}\sum_{u\leq t}\mu(u,t)u=\sum_{u\leq s}u\sum_{u\leq t\leq s}\mu(u,t)\\
 & =\sum_{u\leq s}u\delta(u,s)=s
\end{align*}
and
\[
\varphi\psi(E_{n}(s))=\varphi(\sum_{t\leq s}\mu(t,s)t)=\sum_{t\leq s}\mu(t,s)\varphi(t)=E_{n}(s)
\]
where the last equality follows from the M\"{o}bius inversion theorem
and the definition of $\varphi$. Hence, $\varphi$ and $\psi$ are
bijectives. We now prove that $\varphi$ is a homomorphism.

Let $t,s\in\PT_{n}$ we have to show that
\begin{equation}
\sum_{h\leq ts}E_{n}(h)=(\sum_{t^{\prime}\leq t}E_{n}(t^{\prime}))(\sum_{s^{\prime}\leq s}E_{n}(s^{\prime}))\label{eq:HomomorphismEqualityOfVarphi}.
\end{equation}

\begin{casenv}
\item \label{case:MainIsomorphismProofCase1}First assume that $\dom t=\im s$.
It is clear that for any element $E_{n}(t^{\prime}s^{\prime})$  on the right hand side of (\ref{eq:HomomorphismEqualityOfVarphi}),
$t^{\prime}s^{\prime}$ is less than or equal to $ts$. So we have only to show that any $E_{n}(h)$
for $h\leq ts$ appears in the right hand side once. If $h\leq ts$
one can take $s^{\prime}=\left.s\right|_{\dom h}$ and $t^{\prime}=\left.t\right|_{\im s^{\prime}}$
and it is clear that $E_{n}(t^{\prime})E_{n}(s^{\prime})=E_{n}(h)$.
So $E_{n}(h)$ appears in the right hand side. Now assume that $E_{n}(t^{\prime})E_{n}(s^{\prime})=E_{n}(h)$
for some $t^{\prime}\leq t$ and $s^{\prime}\leq s$ then $\dom h=\dom t^{\prime}s^{\prime}=\dom s^{\prime}$
so $s^{\prime}$ has to be $\left.s\right|_{\dom h}$ and since $E_{n}(t^{\prime})E_{n}(s^{\prime})\neq0$
we know that $t^{\prime}$ has to be $\left.t\right|_{\im s^{\prime}}$.
So $E_{n}(h)$ appears only once.
\item $\dom t\neq\im s$. Choose $\tilde{s}\leq s$ with maximal domain
such that $\im\tilde{s}\subseteq\dom t$ and define $\tilde{t}=\left.t\right|_{\im\tilde{s}}$.
It is clear that $ts=\tilde{t}\tilde{s}$ and $\dom\tilde{t}=\im\tilde{s}$.
Now,\\
\[
\sum_{h\leq ts}E_{n}(h)=\sum_{h\leq\tilde{t}\tilde{s}}E_{n}(h)
\]
and by \caseref{MainIsomorphismProofCase1}
\[
\sum_{h\leq\tilde{t}\tilde{s}}E_{n}(h)=(\sum_{t^{\prime}\leq\tilde{t}}E_{n}(t^{\prime}))(\sum_{s^{\prime}\leq\tilde{s}}E_{n}(s^{\prime})).
\]
If $s^{\prime}\leq s$ but $s^{\prime}\nleq\tilde{s}$ then $\im s^{\prime}\nsubseteq\dom t$
so $E_{n}(t^{\prime})E_{n}(s^{\prime})=0$ for any $t^{\prime}\leq t$.
On the other hand, if $t^{\prime}\leq t$ but $t^{\prime}\nleq\tilde{t}$
then $\dom t^{\prime}\nsubseteq\im s$ so $E_{n}(t^{\prime})E_{n}(s^{\prime})=0$
for any $s^{\prime}\leq s$. Hence,
\[
(\sum_{t^{\prime}\leq\tilde{t}}E_{n}(t^{\prime}))(\sum_{s^{\prime}\leq\tilde{s}}E_{n}(s^{\prime}))=(\sum_{t^{\prime}\leq t}E_{n}(t^{\prime}))(\sum_{s^{\prime}\leq s}E_{n}(s^{\prime}))
\]
and we get the desired equality.
\end{casenv}
\end{proof}
We now want to describe the quiver of the category algebra $\mathbb{C}E_{n}$ and hence of $\mathbb{C}\PT_{n}$.

Margolis and Steinberg \cite[section 6.3.1]{Margolis2012} and Li
\cite{Li2011} independently described the quiver of skeletal EI-categories
and we use their results here. Note that in $E_{n}$, all sets
of the same cardinality are isomorphic objects. Hence, if we denote
by $SE_{n}$ the full subcategory of $E_{n}$ with the objects $\overline{k}=\{1,\ldots,k\}$
where $1\leq k\leq n$ and $\overline{0}=\emptyset$ then $SE_{n}$
is equivalent to $E_{n}$. This implies that their algebras are Morita
equivalent (see \cite[Proposition 2.2]{Webb2007}) and hence have the
same quiver. So we can switch our attention to finding the quiver
of $\mathbb{C}SE_{n}$. Another way to describe $SE_{n}$ is as the
category with object set $\{\overline{k}\mid0\leq k\leq n\}$ and
$SE_{n}(\overline{k},\overline{r})$ is the set of total onto functions
from $\overline{k}$ to $\overline{r}$. We continue to denote the morphism of $SE_n$ associated to some function $t$ by $E_n(t)$. Note that $SE_{n}$ is a
skeletal EI-category and $SE_{n}(\overline{k},\overline{k})\cong S_{k}$. 
\begin{lem}
\label{lem:ClassificationOfIrreducibleMorphisms}The irreducible morphisms
of $SE_{n}$ are precisely the morphisms from $\overline{k+1}$ to
$\overline{k}$. In other words, 
\[
\IRR SE_{n}(\overline{r},\overline{k})=\begin{cases}
SE_{n}(\overline{r},\overline{k}) & r=k+1\\
\emptyset & \mbox{otherwise}
\end{cases}.
\]
\end{lem}
\begin{proof}
It is clear that any morphism from $\overline{k+1}$ to $\overline{k}$ is irreducible. Now, assume that $r>k+1$ and let $E_{n}(t)\in SE_{n}(\overline{r},\overline{k})$
be a morphism, that is, $t$ is a total onto function $t:\overline{r}\to\overline{k}$.
We can choose distinct $a,b\in\overline{r}$ such that $t(a)=t(b)$
and define $s:\overline{r}\to\overline{k+1}$ and $h:\overline{k+1}\to\overline{k}$
by
\[
s(i)=\begin{cases}
t(i) & i\neq a\\
k+1 & i=a
\end{cases}\quad h(i)=\begin{cases}
i & i\leq k\\
t(a) & i=k+1
\end{cases}.
\]
It is clear that $E_{n}(s)$ and $E_{n}(h)$ are morphisms in $SE_{n}$
that are not isomorphisms, but $E_{n}(t)=E_{n}(h)E_{n}(s)$ so $E_{n}(t)$
is not an irreducible morphism.
\end{proof}
Now we can use the following result, which is precisely \cite[Theorem 6.13]{Margolis2012}
and \cite[Theorem 4.7]{Li2011} for the case of the field of complex
numbers.
\begin{thm}
\label{thm:QuiverOfEICategories}Let $A$ be a finite skeletal EI-category
and denote by $Q$ the quiver of $\mathbb{C}A$. Then:\end{thm}
\begin{enumerate}
\item The vertex set of $Q$ is ${\displaystyle \bigsqcup_{c\in A^{0}}}\Irr A(c,c)$.
\item If $V\in\Irr(A(c,c))$ and $U\in\Irr(A(c^{\prime},c^{\prime}))$,
then the number of arrows from $V$ to $U$ is the multiplicity of
$U\otimes V^{\ast}$ as an irreducible constituent in the $A(c^{\prime},c^{\prime})\times A(c,c)$-module $\mathbb{C}\IRR A(c,c^{\prime})$. Where the operation on $\mathbb{C}\IRR A(c,c^{\prime})$
is given by $(h,g)\ast f=hfg^{-1}$.
\end{enumerate}
Applying \thmref{QuiverOfEICategories} to our case enables us to
translate our original question to a problem in the theory of representations
of the symmetric group. In our case the endomorphism groups are $S_{k}$
for $0\leq k\leq n$ hence the vertex set is ${\displaystyle \bigsqcup_{k=0}^{n}}\Irr S_{k}$.
If $V\in\Irr(S_{k})$ and $U\in\Irr(S_{r})$ are such that $k\neq r+1$
then there are no arrows from $V$ to $U$ since by \lemref{ClassificationOfIrreducibleMorphisms}
there are no irreducible morphisms between the corresponding objects
in $SE_{n}$. If $U\in\Irr(S_{k})$ and $V\in\Irr(S_{k+1})$ then
the number of arrows from $V$ to $U$ is the multiplicity of $U\otimes V^{\ast}$
as an irreducible constituent in the $S_{k}\times S_{k+1}$-module
$M$, where $M$ is spanned by all the onto function $f:\overline{k+1}\to\overline{k}$
and the operation is $(h,g)\ast f=hfg^{-1}$. Note that $M$ is a
permutation module of $S_{k}\times S_{k+1}$ with basis $X=\{f:\overline{k+1}\to\overline{k}\mid f\text{ is onto}\}$.
The action on $X$ is transitive so if we choose any $f\in X$ and
denote by $K=\Stab(f)$ its stabilizer then: 
\[
M=\Ind_{K}^{S_{k}\times S_{k+1}}\tr_{K}
\]
 where $\tr_{K}$ is the trivial module of $K$. Now, choose $f\in X$
to be 
\[
f(i)=\begin{cases}
i & i\leq k\\
k & i=k+1
\end{cases}.
\]
 Let us describe $K$ more explicitly.
\begin{lem}
$K=\{(\sigma,\sigma\tau)\mid\sigma\in S_{k-1},\quad\tau\in S_{\{k,k+1\}}\}$.\end{lem}
\begin{proof}
Assume $h\in S_{k}$ and $g\in S_{k+1}$ are such that $hfg=f$ .
Now, denote $l_1=g^{-1}(k)$ and $l_2=g^{-1}(k+1)$. Since $hfg(l_1)=hfg(l_2)$
and $hfg=f$ then we must have that $l_1=k$ and $l_2=k+1$ or vice versa.
In other words $g$ must send $\{k,k+1\}$ onto $\{k,k+1\}$. For
$i<k$ if $g(i)=j$ then we must have $h(j)=i$. Now it is clear that
there are $\sigma\in S_{k-1}$ and $\tau\in S_{\{k,k+1\}}$ such that
$g=\sigma\tau$ and $h=\sigma^{-1}$ (we view $S_{k-1}$ as a subgroup
of $S_{k}$ in the usual way). Since our action is $(h,g)\ast f=hfg^{-1}$
we see that the stabilizer of $f$ is $K=\Stab(f)=\{(\sigma,\sigma\tau)\mid\sigma\in S_{k-1},\quad\tau\in S_{\{k,k+1\}}\}\cong S_{k-1}\times S_{2}$.
\end{proof}
In the following computations we will write $S_{2}$ instead of $S_{\{k,k+1\}}$
and regard it as a subgroup of $S_{k+1}$. Also we regard in the usual
way $S_{k-1}\times S_{2}$ and $S_{k-1}$ as subgroups of $S_{k+1}$
and $S_{k}$ respectively. We also denote by $\tr_{2}$ the trivial
representation of $S_{2}$.
\begin{lem}
\label{lem:AlternativeDescriptionOfNumberOfArrows}The number of arrows
from $V$ to $U$ is the multiplicity of $V$ as an irreducible constituent
in the $S_{k+1}$-module $\Ind_{S_{k-1}\times S_{2}}^{S_{k+1}}(\Res_{S_{k-1}}^{S_{k}}(U)\otimes\tr_{2})$.\end{lem}
\begin{proof}
The number of arrows from $V$ to $U$ is the multiplicity of $U\otimes V^{\ast}$
in $M$ and this number can be expressed by the inner product of characters:
\[
\langle U\otimes V^{\ast},\Ind_{K}^{S_{k}\times S_{k+1}}\tr_{K}\rangle
\]

(recall that in order to simplify notation, we use the same notation
for the representation and its character). Using Frobenius reciprocity
we can see that:
\begin{align*}
\langle U\otimes V^{\ast},\Ind_{K}^{S_{k}\times S_{k+1}}\tr_{K}\rangle & =\langle\Res_{K}^{S_{k}\times S_{k+1}}(U\otimes V^{\ast}),\tr_{K}\rangle\\
 & =\frac{1}{|K|}\sum_{(\sigma,\sigma\tau)\in K}U\otimes V^{\ast}((\sigma,\sigma\tau))\\
 & =\frac{1}{|K|}\sum_{(\sigma,\tau)\in S_{k-1}\times S_{2}}U(\sigma)V^{\ast}(\sigma\tau).\\
\end{align*}

Since the characters of $S_{n}$ are real-valued, $V^{\ast}(\sigma\tau)=V(\sigma\tau)$
so this equals:
\begin{align*}
\frac{1}{|K|}\sum_{(\sigma,\tau)\in S_{k-1}\times S_{2}}U(\sigma)V(\sigma\tau) & =\frac{1}{|K|}\sum_{(\sigma,\tau)\in S_{k-1}\times S_{2}}V(\sigma\tau)U(\sigma)\tr_{2}(\tau)\\
 & =\langle\Res_{K}^{S_{k+1}}V,\Res_{S_{k-1}}^{S_{k}}(U)\otimes\tr_{2}\rangle.
\end{align*}

Again, using Frobenius reciprocity this equals:
\[
\langle V,\Ind_{S_{k-1}\times S_{2}}^{S_{k+1}}(\Res_{S_{k-1}}^{S_{k}}(U)\otimes\tr_{2})\rangle.
\]

\end{proof}
The benefit of the description of \lemref{AlternativeDescriptionOfNumberOfArrows}
is that the module \linebreak $\Ind_{S_{k-1}\times S_{2}}^{S_{k+1}}(\Res_{S_{k-1}}^{S_{k}}(U)\otimes\tr_{2})$
has a good combinatorial description using Young diagrams. We will
use facts from \cite[section 2.8]{James1981}. Recall that if $\alpha$
is the Young diagram corresponding to $W\in\Irr S_{k}$ then 
\[
\Res_{S_{k-1}}^{S_{k}}(W)
\]

is the sum of simple modules that correspond to the diagrams that
are obtained from $\alpha$ by removing one box (this is the well
known branching rule). Now, if $\alpha$ is the Young diagram that
corresponds to $W\in S_{k-1}$ then the module 
\[
\Ind_{S_{k-1}\times S_{2}}^{S_{k+1}}(W\otimes\tr_{2})
\]

is the sum of simple modules that correspond to the diagrams that
are obtained from $\alpha$ by adding two boxes, but not in the same
column. This is a special case of Young's rule.

Hence, if $U$ is an irreducible $S_{k}$-module that corresponds
to a Young diagram $\alpha$ then the $S_{k+1}$-module $\Ind_{S_{k-1}\times S_{2}}^{S_{k+1}}(\Res_{S_{k-1}}^{S_{k}}(U)\otimes\tr_{2})$
corresponds to the sum of Young diagrams obtained from $\alpha$ by
removing one box and then adding two boxes but not in the same column. A diagram can appear in this summation more than once and we count the diagrams with multiplicity. If $V$ corresponds to a Young diagram $\beta$ then the number of arrows
from $V$ to $U$ is the number of times that $\beta$ occur in this
summation.
\begin{example}
Let $U$ be the standard representation of $S_{3}$ whose corresponding
Young diagram is $\alpha=[2,1]$: 
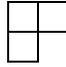
\[
\yng(2,1)
\]

Then the module 
\[
\Res_{S_{k-1}}(U)
\]
corresponds to 
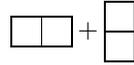
\[
\yng(2)+\yng(1,1)
\]
where sum of diagrams means the direct sum of the corresponding simple
modules. Now, the module
\[
\Ind_{S_{k-1}\times S_{2}}^{S_{k+1}}(\Res_{S_{k-1}}^{S_{k}}(U)\otimes\tr_{2})
\]
corresponds to
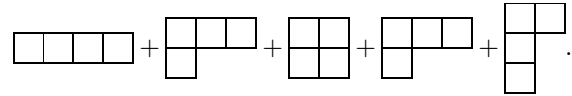
\[
\yng(4)+\yng(3,1)+\yng(2,2)+\yng(3,1)+\yng(2,1,1).
\]
Hence, there are two arrows in the quiver from $\tiny\yng(3,1)$ to
$\tiny\yng(2,1)$, and one arrow from $\tiny\yng(4)$, $\tiny\yng(2,2)$
and $\tiny\yng(2,1,1)$ to $\tiny\yng(2,1)$. A full drawing of the
quiver of $\mathbb{C}\PT_{4}$ is given in the next figure:
\end{example}
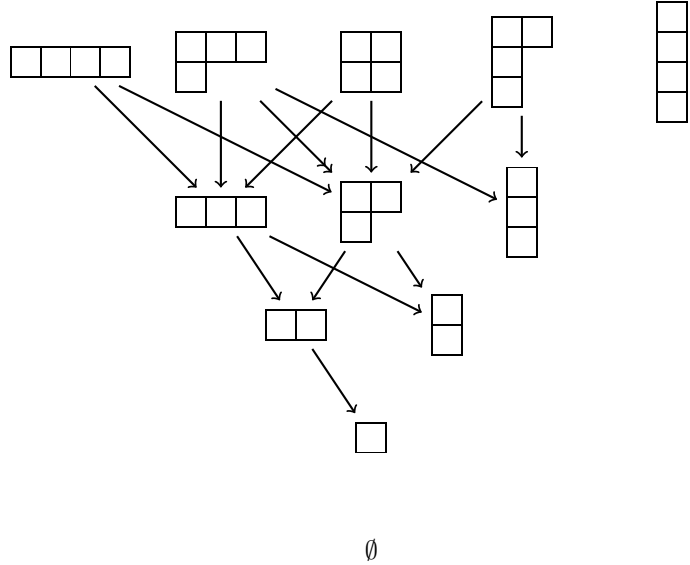
\begin{center} \begin{tikzpicture}\path (-4,5) node (S4_1) {$\yng(4)$};\path (-2,5) node (S4_2) {$\yng(3,1)$};\path (0,5) node (S4_3) {$\yng(2,2)$};\path (2,5) node (S4_4) {$\yng(2,1,1)$};\path (4,5) node (S4_5) {$\yng(1,1,1,1)$}; \path (-2,3) node (S3_1) {$\yng(3)$}; \path (0,3) node (S3_2) {$\yng(2,1)$}; \path (2,3) node (S3_3) {$\yng(1,1,1)$}; \path (-1,1.5) node (S2_1) {$\yng(2)$}; \path (1,1.5) node (S2_2) {$\yng(1,1)$}; \path (0,0) node (S1_1) {$\yng(1)$}; \path (0,-1.5) node (S0_1) {$\emptyset$};\draw[thick,->] (S4_1)--(S3_1) ;\draw[thick,->] (S4_2)--(S3_1) ;\draw[thick,->] (S4_3)--(S3_1) ; \draw[thick,->] (S4_1)--(S3_2) ;\draw[thick,->>] (S4_2)--(S3_2) ;\draw[thick,->] (S4_3)--(S3_2) ;\draw[thick,->] (S4_4)--(S3_2) ; \draw[thick,->] (S4_2)--(S3_3) ; \draw[thick,->] (S4_4)--(S3_3) ; \draw[thick,->] (S3_1)--(S2_1) ; \draw[thick,->] (S3_2)--(S2_1) ; \draw[thick,->] (S3_1)--(S2_2) ; \draw[thick,->] (S3_2)--(S2_2) ; \draw[thick,->] (S2_1)--(S1_1) ;\end{tikzpicture} \end{center}In
conclusion, we end up with the following theorem:
\begin{thm}
\label{thm:DescriptionOfTheQuiverOfPT_n}The vertices in the quiver
of $\mathbb{C}\PT_{n}$ are in one-to-one correspondence with Young
diagrams with $k$ boxes where $0\leq k\leq n$. If $\alpha\vdash k$,
\textup{$\beta\vdash r$} are two Young diagrams such that $r\neq k+1$
then there are no arrows from $\beta$ to $\alpha$. If $r=k+1$ then
there are arrows from $\beta$ to $\alpha$ if we can construct $\beta$
from $\alpha$ by removing one box and then adding two boxes but not
in the same column. The number of arrows is the number of different
ways that this construction can be carried out.\end{thm}
\begin{rem}
Note that up to rank $n-1$ the quiver of $\mathbb{C}\PT_{n}$ is
precisely the quiver of $\mathbb{C}\PT_{n-1}$.
\end{rem}

\section{Other invariants of $\mathbb{C}\PT_{n}$}

In this section we use the above results in order to find other important
invariants of the algebra $\mathbb{C}\PT_{n}$.

\subsection{Connected components of the quiver of $\mathbb{C}\PT_{n}$}

\begin{prop}
For every $n>1$, the quiver of $\mathbb{C}\PT_{n}$ has three connected
components, with two isolated components: $\emptyset$ and $[1^{n}]$
(the sign representation of rank $n$).\end{prop}
\begin{proof}
Using \thmref{DescriptionOfTheQuiverOfPT_n}, it is easy to prove
the statement by induction. The claim is obvious for $n=2$. For $n>2$,
assume that the quiver of $\mathbb{C}\PT_{n-1}$ has three connected components
with $[1^{n-1}]$ and $\emptyset$ being isolated. Now consider the
quiver of $\mathbb{C}\PT_{n}$. Recall that up to rank $n-1$ the
quiver of $\mathbb{C}\PT_{n}$ is the quiver of $\mathbb{C}\PT_{n-1}$.
Moreover, it is clear that there is no arrow from $[1^{n}]$ and there
exists an arrow from $[3,1^{n-3}]$ to $[1^{n-1}]$. Hence it is left
to show that the rank $n$ vertices (except for $[1^{n}]$) are connected
to some rank $n-1$ representation (other than $[1^{n-1}]$). Consider
some vertex $v=[\alpha_{1},\ldots,\alpha_{k}]$ and note that $\alpha_{1}>1$
(since $v\neq[1^{n}]$). If $\alpha_{k}>1$ then there is an arrow
from $v$ to $u=[\alpha_{1},\ldots,\alpha_{k-1},(\alpha_{k}-1)]$  because one
can remove one box from the last row and add two. The next figure illustrates this case:

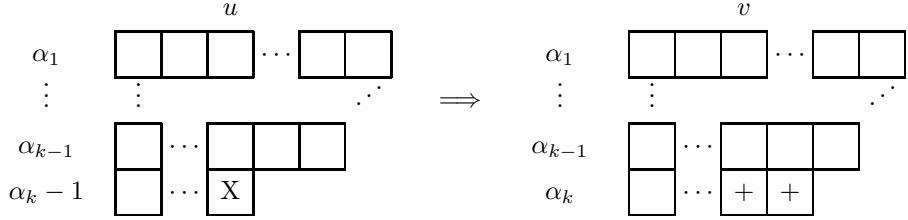
\ytableausetup
{mathmode, boxsize=1.7em}
\begin{ytableau}
\none & \none & \none & \none & \none[u]\\
\none[\alpha_1] & \none & \quad & \quad   & \quad & \none[\cdots] & \quad & \quad  \\
\none[\vdots] & \none & \none[\vdots] & \none & \none   & \none  & \none & \none[\iddots] & \none & \none[\Longrightarrow] & \none\\
 \none[ \alpha_{k-1}] & \none & \quad & \none[\cdots] & \quad & \quad & \quad \\
 \none[ \alpha_k-1] & \none & \quad & \none[\cdots] & \text{X}
\end{ytableau} 
\begin{ytableau}
\none & \none & \none & \none & \none[v]\\
\none[\alpha_1] & \none & \quad & \quad   & \quad & \none[\cdots] & \quad & \quad \\
\none[\vdots] & \none & \none[\vdots] & \none & \none & \none & \none & \none[\iddots]  \\
 \none[ \alpha_{k-1}] & \none & \quad & \none[\cdots] & \quad & \quad & \quad \\
 \none[ \alpha_k] & \none & \quad & \none[\cdots] & \bf{+} & \bf{+}
\end{ytableau}
\newline
\newline
The box removed from $u$ is marked with "X" and the boxes added to obtain $v$  are marked with "+".
Now, if $\alpha_{k}=1$
then we can write $v=[\alpha_{1},\ldots,\alpha_{k-1},1]$. Let $l$
be maximal such that $\alpha_{l}>1$ (such $l$ exists since $\alpha_{1}>1$).
It is easy to observe that there is an arrow from $v$ to $u=[\alpha_{1},\ldots,\alpha_{k-1}]$. This is because we can remove a box from the $l$-th row and add
two boxes, one in the $l$-th row and one in the last row. The requirement that $l$ is maximal such that $\alpha_{l}>1$ ensures that they are not in the same column. This case is illustrated in the next figure:

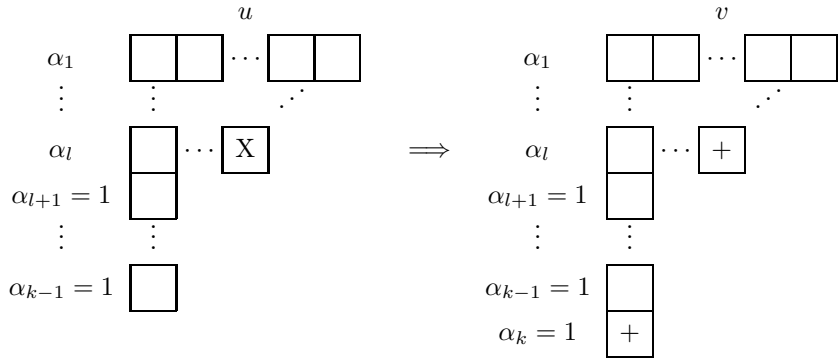
\ytableausetup
{mathmode, boxsize=1.7em}
\begin{ytableau}
\none & \none & \none & \none & \none[u]\\
\none[\alpha_1] & \none &  \quad   & \quad & \none[\cdots] & \quad & \quad  \\
\none[\vdots] & \none & \none[\vdots] & \none & \none    & \none[\iddots] \\
 \none[ \alpha_{l}] & \none & \quad & \none[\cdots] & \text{X}  & \none & \none & \none & \none[\Longrightarrow] & \none\\
 \none[\alpha_{l+1}=1] & \none & \quad  \\
\none[\vdots] & \none & \none[\vdots] \\
 \none[\alpha_{k-1}=1] & \none & \quad 
\end{ytableau} 
\ytableausetup
{mathmode, boxsize=1.7em}
\begin{ytableau}
\none & \none & \none & \none & \none[v]\\
\none[\alpha_1] & \none &  \quad   & \quad & \none[\cdots] & \quad & \quad  \\
\none[\vdots] & \none & \none[\vdots] & \none & \none  & \none[\iddots] \\
 \none[ \alpha_{l}] & \none & \quad & \none[\cdots] & \bf{+}  \\
 \none[\alpha_{l+1}=1] & \none & \quad  \\
\none[\vdots] & \none & \none[\vdots] \\
 \none[\alpha_{k-1}=1] & \none & \quad \\ 
 \none[\alpha_{k}=1] & \none & \bf{+}
\end{ytableau} 
\newline
\newline
We have marked the removed and added boxes with "X" and "+" respectively as above.
This finishes the proof.
\end{proof}

\subsection{Loewy series of $\mathbb{C}\PT_{n}$}

Now we observe that we can ``see'' $\Rad^{k}\mathbb{C}E_{n}$ inside
the category itself. In other words, certain morphisms of the category
$E_{n}$ span $\Rad^{k}\mathbb{C}E_{n}$. We start with the case $k=1$. 

We mention that \cite[Proposition 4.6]{Li2011} is a similar observation
for any EI-category.
\begin{lem}
\label{lem:RadicalOfEpimorphismCategory}$\Rad\mathbb{C}E_{n}=\Span\{E_{n}(t)\mid t\in\PT_{n}\backslash\IS_{n}\}$.\end{lem}
\begin{proof}
Write $R=\Span\{E_{n}(t)\mid t\in\PT_{n}\backslash\IS_{n}\}$. It
is easy to see that $R$ is a nilpotent ideal hence $R\subseteq\Rad\mathbb{C}E_{n}$.
In addition $\mathbb{C}E_{n}/R\cong\mathbb{C}G_{n}$, but $G_{n}$
is a groupoid and its algebra is semisimple, hence $\Rad\mathbb{C}E_{n}\subseteq R$
and we are done.\end{proof}
\begin{lem}
\label{lem:DescriptionOfRadToTheK}$\Rad^{k}\mathbb{C}E_{n}=\Span\{E_{n}(t)\mid|\dom t|-|\im t|\geq k\}$.\end{lem}
\begin{proof}
Clearly $\Rad^{k}\mathbb{C}E_{n}\subseteq\Span\{E_{n}(t)\mid|\dom t|-|\im t|\geq k\}$.
So it suffices to show the other inclusion. Now, take $t$ such that
$|\dom t|-|\im t|\geq k$. It is enough to show that $E_{n}(t)$ can
be written as a product of $k$ elements from $\{E_{n}(t)\mid|\dom t|-|\im t|\geq1\}$
which is a basis for $\Rad\mathbb{C}E_{n}$. This is easily done by
induction. The case $k=1$ is trivial. Now, choose two distinct elements
$a$ and $a^{\prime}$ from $\dom t$ such that $t(a)=t(a^{\prime})$
and choose $b\notin\im t$. We can write $E_{n}(t)$ as a product
$E_{n}(t)=E_{n}(h)E_{n}(s)$ where
\[
s(i)=\begin{cases}
t(i) & i\neq a\\
b & i=a
\end{cases}\quad h(i)=\begin{cases}
i & i\neq b\\
t(a) & i=b
\end{cases}.
\]
Note that $|\dom h|-|\im h|=1$ and $|\dom s|-|\im s|\geq k-1$ so
by the induction hypothesis we are done.
\end{proof}
 \cref{lem:DescriptionOfRadToTheK} immediately implies the following corollary:
\begin{cor}
The Loewy length of $\mathbb{C}\PT_n$ is $n$.
\end{cor}
We can also use \lemref{DescriptionOfRadToTheK} to get a formula for the
dimension of $\Rad^{k}\mathbb{C}\PT_{n}$. Recall that the \emph{Stirling number of the second kind} $S(d,m)$ is the number of ways to partition a set of $d$ objects into $m$ non-empty subsets.
\begin{lem}
$\dim\Rad^{k}\mathbb{C}\PT_{n}$ equals
\[
\sum_{d=k+1}^{n}\sum_{m=1}^{d-k}{n \choose d}{n \choose m}S(d,m)m!.
\]
\end{lem}
\begin{proof}
We have only to count the basis elements given in \lemref{DescriptionOfRadToTheK}.
The number of total functions with domain of size $d$ and image of
size $m$ is 
\[
S(d,m)m!
\]
since one has $S(d,m)$ different ways to partition the domain into
$m$ non-empty subsets and then $m!$ ways to match the subsets with
image elements. There are ${n \choose d}{n \choose m}$ ways to choose
domain of size $d$ and image of size $m$ so all that is left to do is
to sum all possible sizes of the domain and image.
\end{proof}

\section{Quivers of submonoids of $\PT_{n}$ which are order ideals}
In this section we apply the method we have used to describe the quiver
of $\mathbb{C}\PT_{n}$ in order to find the quiver of the algebra
of other well known transformation monoids. All monoids discussed in this section are extensively studied in \cite[Chapter 14]{Ganyushkin2009b}. The important observation
is the following one: Let $N$ be a submonoid of $\PT_{n}$
that is also an order ideal, that is, if $y\in N$ and $x\leq y$
then $x\in N$. Let $D_{n}$ be the subcategory
of $E_{n}$ with the same set of objects and morphism set $\{E_{n}(t)\mid t\in N\}$.
Note that since $N$ is an order ideal we have $\varphi(N)\subseteq \mathbb{C}D_n$ and $\psi(D_n)\subseteq \mathbb{C}N$. Hence the restriction of $\varphi$ to $\mathbb{C}N$ gives
an isomorphism with $\mathbb{C}D_{n}$. $D_{n}$ is also an EI-category so we can again use \thmref{QuiverOfEICategories}
in order to compute the quiver of its algebra.

\subsection{Order-preserving partial functions}

Let $\PO_{n}$ be the monoid of all order-preserving partial functions on $\overline{n}$,
that is, 
\[
\PO_{n}=\{t\in\PT_{n}\mid\forall x,y\in\dom t\quad x\leq y\Rightarrow t(x)\leq t(y)\}.
\]
$\PO_{n}$ is indeed a submonoid of $\PT_{n}$ and an ideal with
respect to inclusion. So we get an isomorphism
\[
\mathbb{C}\PO_{n}\cong\mathbb{C}EO_{n}
\]
where $EO_{n}$ is the subcategory of $E_{n}$ with the same set of objects but whose only morphisms are
$E_{n}(t)$ for $t\in\PO_{n}$. As before we can take the skeleton
of $EO_{n}$ which is equivalent to $EO_{n}$ hence their
algebras have the same quiver. The skeleton will be denoted $SEO_{n}$. Its set of objects is
$\{\overline{k}\mid0\leq k\leq n\}$ and $SEO_{n}(\overline{k},\overline{r})$ are all the order-preserving functions from $\overline{k}$ onto $\overline{r}$. We continue to denote the morphism of $SEO_n$ associated to the function $t$ by $E_n(t)$. Similar to \lemref{ClassificationOfIrreducibleMorphisms} we have the following lemma.
\begin{lem}
\[
\IRR SEO_{n}(\overline{r},\overline{k})=\begin{cases}
SEO_{n}(\overline{r},\overline{k}) & r=k+1\\
\emptyset & \mbox{otherwise}
\end{cases}.
\]
\end{lem}
\begin{proof}
It is clear that any morphism from $\overline{k+1}$ to $\overline{k}$ is irreducible. Now, assume that $r>k+1$ and let $E_{n}(t)\in SEO_{n}(\overline{r},\overline{k})$
be a morphism, that is, $t$ is a total onto order-preserving function $t:\overline{r}\to\overline{k}$. Choose some $b\in \overline{k}$ whose preimage $t^{-1}(b)$ contains more than one element and let $a$ be the maximal element in $t^{-1}(b)$.
Define $s:\overline{r}\to\overline{k+1}$ and $h:\overline{k+1}\to\overline{k}$
by
\[
s(i)=\begin{cases}
t(i) & i < a\\
t(i)+1 & i\geq a
\end{cases}\quad h(i)=\begin{cases}
i & i\leq b \\
i-1 & i> b
\end{cases}.
\]
It is clear that $E_{n}(s)$ and $E_{n}(h)$ are morphisms in $EO_{n}$
that are not isomorphisms, but $E_{n}(t)=E_{n}(h)E_{n}(s)$ so $E_{n}(t)$
is not an irreducible morphism.

\end{proof}

Note that the all the endomorphism groups of $SEO_{n}$ are trivial
and the number of order-preserving functions from $\overline{k+1}$
onto $\overline{k}$ is $k$. Using \thmref{QuiverOfEICategories}
we can conclude:
\begin{prop}
The vertex set of the quiver of $\mathbb{C}\PO_{n}$ is $\{0,\ldots,n\}$. There are $k$ arrows from $k+1$ to $k$, for $k=0,\cdots, n-1$, and no other arrows.
\end{prop}

\subsection{Order-decreasing partial and total functions}

A function $t\in \PT_n$ is called \emph{order-decreasing} if $t(x)\leq x$ for every $x\in \dom(t)$.
In this section we want to consider the monoids of all total and partial order-decreasing functions on $\overline{n}$ denoted by $\F_n$ and $\PF_n$  respectively. We start by proving that these two families are in fact identical.

\begin{lem}\label{lem:IsoIncIPT}
$\F_{n+1} \cong  \PF_n$.
\end{lem}
\begin{proof}
In this proof it will be more convenient to identify $\F_{n+1}$ with the set of order-decreasing total functions on $\{0,\ldots,n\}$.
Note that any $t\in \F_{n+1}$ must satisfy $t(0)=0$ so we can define $f:\F_{n+1}\to\PF_n $ by
\[
f(t)(i)=\begin{cases}
t(i) & t(i)\neq 0 \\
\text{undefined} & t(i)=0
\end{cases}.
\]
Conversely, given $s\in\PF_n$ we define $g: \PF_n \to \F_{n+1} $ by
\[
g(s)(i)=\begin{cases}
s(i) & i\in \dom(s) \\
0 & \text{otherwise}
\end{cases}.
\]
Clearly $f$ and $g$ are monoid homomorphisms and inverse to each other so we get the required isomorphism.
\end{proof}

One reason for the importance of $\F_n$ is the following result. A monoid $M$ is $\Lc$-trivial (that is, any
two distinct elements generate different left ideals) if and only if it is isomorphic to a submonoid of $\F_n$ for some $n\in\mathbb{N}$ \cite[Theorem 3.6]{Pin1986}. In particular, $\F_n$ is an $\Lc$-trivial monoid. We now turn to computing the quivers of $\mathbb{C}\PF_n$ and $\mathbb{C}\F_n$. Note that $\PF_n$ and $\F_n$ are not regular monoids, since in a regular
$\Lc$-trivial monoid every element is an idempotent. But this does
not prevent us from applying our method.
Clearly $\PF_n$ is a submonoid of $\PT_n$ which is an order ideal so we have an isomorphism
\[
\mathbb{C}\PF_{n}\cong\mathbb{C}EF_{n}
\]
where $EF_{n}$ is the subcategory of $E_{n}$ with the same set of objects but whose only morphisms are
$E_{n}(t)$ for $t\in\PF_{n}$. 

Note that for any $X\subseteq\overline{n}$, there is only one order-decreasing function $t$ such that $\dom t=\im t=X$, so every endomorphism
group in $EF_{n}$ is trivial. Moreover, if $X,Y\subseteq\overline{n}$
where $|X|=|Y|$ and $X\neq Y$ then one of $EF_{n}(X,Y)$ and $EF_{n}(Y,X)$
has to be empty. Hence, there are no distinct isomorphic objects in
$EF_{n}$ so $EF_{n}$ is its own skeleton. Hence, \thmref{QuiverOfEICategories}
implies that the vertices of the quiver of $\mathbb{C}EF_{n}$ are
precisely the objects of $EF_{n}$ and the morphisms are precisely
the irreducible morphisms. All that is left to do is to identify the irreducible
morphisms.

In order to do so, we introduce a technical definition. Let $X\subseteq\overline{n}$
and let $j\in X$. Define $j_{X}^{-}$ to be
\[
j_{X}^{-}=\max\{x\in\overline{n}\mid x\notin X,\quad x<j\}
\]

and if such a maximum does not exist then $j_{X}^{-}=1$. Note that
if $j_X^{-}<x<j$ then $x\in X$.

In the following $X$ will always be $\dom t$ so we will usually
omit it and write $j^{-}$ instead of $j_{X}^{-}$. Now we can state
and prove the following result.
\begin{lem}\label{lem:IrreduciblesInEF_n}
$E_{n}(t)$ is irreducible in $EF_{n}$ if and only if there exists
$j\in\dom t$ such that $t(i)=i$ for any $i\in\dom t\backslash\{j\}$
and $j_{X}^{-}\leq t(j)$ where $X=\dom t$. \end{lem}
\begin{proof}
First assume that $E_{n}(t)$ is irreducible. Let $j$ be maximal
in $\dom t$ such that $t(j)<j$ (such $j$ must exist since $E_{n}(t)$
is not an isomorphism). If there is another $j^{\prime}\in\dom t$
such that $t(j^{\prime})<j^{\prime}$ then we can define $s,h\in\PF_{n}$
by
\[
s(i)=\begin{cases}
t(i) & i\in\dom t\backslash\{j\}\\
j & i=j
\end{cases}\quad h(i)=\begin{cases}
i & i\in\im s\backslash\{j\}\\
t(j) & i=j
\end{cases}.
\]

It is easy to observe that $s,h\in\PF_{n}$. We have already seen that the only isomorphisms of $EF_n$ are the identity morphisms. $E_{n}(h)$ and $E_{n}(s)$ are not isomorphisms because $s(j^\prime)=t(j^\prime)<j^\prime$ and $h(j)=t(j)<j$. Since $E_{n}(h)E_{n}(s)=E_{n}(t)$ we get a contradiction. So there is only one $j\in\dom t$ such that $t(j)<j$.
Now assume that $t(j)<j^{-}$. Note that this implies that $j^{-}\notin\dom t$
since $j^{-}\neq 1$. We can define
\[
s(i)=\begin{cases}
i & i\in\dom t\backslash\{j\}\\
j^{-} & i=j
\end{cases}\quad h(i)=\begin{cases}
i & i\in\im s\backslash\{j^{-}\}\\
t(j) & i=j^{-}
\end{cases}.
\]

Again, $E_{n}(h)$ and $E_{n}(s)$ are clearly not isomorphisms. It is easy to see that $s,h\in\PF_{n}$ and $E_{n}(h)E_{n}(s)=E_{n}(t)$
which contradicts the assumption and ends this direction. In the other
direction, assume that $t$ is of the required form but $E_{n}(t)=E_{n}(h)E_{n}(s)$
where $E_{n}(h)$ and $E_{n}(s)$ are not isomorphisms. Since for
any $i\in\dom t\backslash\{j\}$ we have $hs(i)=t(i)=i$ we must have that
$h(i)=s(i)=i$. Now, since $s$ and $h$ are not the identity on their domains,
we must have $ j^{-}\leq t(j)=h(s(j))<s(j)<j$. But this implies that
$s(j)\in\dom t\backslash\{j\}$ hence $hs(j)=s(j)\neq t(j)$, a contradiction.
\end{proof}
The next result now follows immediately.
\begin{prop}
The vertices in the quiver of the algebra $\mathbb{C}\PF_{n}$ are
in one-to-one correspondence with subsets of $\overline{n}$.
For $X,Y\subseteq\overline{n}$, the arrows from $X$ to $Y$ are
in one-to-one correspondence with onto functions $t:X\to Y$ for which
there exists $j\in X$ such that $t(i)=i$ for $i\in X\backslash\{j\}$
and $j_X^{-}\leq t(j)<j$.
\end{prop}

Using \lemref{IsoIncIPT} we get a description for the quiver of $\F_n$ as well.

\begin{cor}
The vertices in the quiver of the algebra $\mathbb{C}\F_{n}$ are
in one-to-one correspondence with subsets of $\overline{n-1}$ (where $\overline{0}=\emptyset$).
For $X,Y\subseteq\overline{n-1}$, the arrows from $X$ to $Y$ are
in one-to-one correspondence with onto functions $t:X\to Y$ for which
there exists $j\in X$ such that $t(i)=i$ for $i\in X\backslash\{j\}$
and $j_X^{-}\leq t(j)<j$.
\end{cor}

\subsection{Partial Catalan monoid}

Define $\PC_n$, called the \emph{partial Catalan monoid}, to be the monoid of all partial function on $\overline{n}$ which are both order-preserving and order-decreasing. 
The computation of the quiver of $\mathbb{C}\PC_n$ is quite similar to that of $\mathbb{C}\PF_n$. $\PC_n$ is indeed a submonoid of $\PT_n$ and an order ideal. We get an isomorphism 
\[
\mathbb{C}\PC_n \cong \mathbb{C}EC_n
\]
where $EC_{n}$ is the subcategory of $E_{n}$ with the same set of objects but whose only morphisms are $E_{n}(t)$ for $t\in\PC_{n}$. Note that $EC_{n}$ is obtained from $EF_n$  by erasing morphisms so it is clear that it has no isomorphic objects and all the endomorphism groups are trivial. So again we just have to identify the irreducible morphisms.

\begin{lem}
$E_{n}(t)$ is irreducible in $EC_{n}$ if and only if there exists
$j\in\dom t$ such that $t(i)=i$ for any $i\in\dom t\backslash\{j\}$
and $t(j)=j-1$. \end{lem}
\begin{proof}
Assume that $t$ is of the required form. By \lemref{IrreduciblesInEF_n} $E_n(t)$ is irreducible in $EF_n$ so it must be irreducible in $EC_n$ as well.
In the other direction, we can prove precisely as in \lemref{IrreduciblesInEF_n} that there is a unique $j\in\dom t$ such that $t(j)<j$.
Now assume that $t(j)<j-1$. Note that there is no $k\in \dom (t)$ such that $t(j)<k<j$ because this will imply that $t(j)<k=t(k)$ in contradiction to the fact that $t$ is order-preserving.  Define
\[
s(i)=\begin{cases}
i & i\in\dom t\backslash\{j\}\\
j-1 & i=j
\end{cases}\quad h(i)=\begin{cases}
i & i\in\im s\backslash\{j-1\}\\
t(j) & i=j-1
\end{cases}.
\]

Again, $D_{n}(h)$ and $D_{n}(s)$ are clearly not isomorphisms. It is easy to see that $s,h\in \PC_{n}$ and $E_{n}(h)E_{n}(s)=E_{n}(t)$
which contradicts the assumption and ends the proof.
\end{proof}

We conclude:

\begin{prop}
The vertices in the quiver of the algebra $\mathbb{C}\PC_{n}$ are
in one-to-one correspondence with subsets of $\overline{n}$.
For $X,Y\subseteq\overline{n}$, the arrows from $X$ to $Y$ are
in one-to-one correspondence with onto functions $t:X\to Y$ for which
there exists $j\in X$ such that $t(i)=i$ for $i\in X\backslash\{j\}$
and $t(j)=j-1$.
\end{prop}
{\bf Acknowledgement:} The author is grateful to the referee for his\textbackslash  her valuable comments and suggestions, and in particular for suggesting \lemref{IsoIncIPT}.

\bibliographystyle{plain}
\bibliography{library}

\begin{thebibliography}{10}

\bibitem{Almeida2009}
Jorge Almeida, Stuart Margolis, Benjamin Steinberg, and Mikhail Volkov.
\newblock Representation theory of finite semigroups, semigroup radicals and
  formal language theory.
\newblock {\em Trans. Amer. Math. Soc.}, 361(3):1429--1461, 2009.

\bibitem{Alperin1995}
J.~L. Alperin and Rowen~B. Bell.
\newblock {\em Groups and representations}, volume 162 of {\em Graduate Texts
  in Mathematics}.
\newblock Springer-Verlag, New York, 1995.

\bibitem{Assem2006}
Ibrahim Assem, Daniel Simson, and Andrzej Skowro{\'n}ski.
\newblock {\em Elements of the representation theory of associative algebras.
  {V}ol. 1}, volume~65 of {\em London Mathematical Society Student Texts}.
\newblock Cambridge University Press, Cambridge, 2006.
\newblock Techniques of representation theory.

\bibitem{Clifford1961}
A.~H. Clifford and G.~B. Preston.
\newblock {\em The algebraic theory of semigroups. {V}ol. {I}}.
\newblock Mathematical Surveys, No. 7. American Mathematical Society,
  Providence, R.I., 1961.

\bibitem{Clifford1967}
A.~H. Clifford and G.~B. Preston.
\newblock {\em The algebraic theory of semigroups. {V}ol. {II}}.
\newblock Mathematical Surveys, No. 7. American Mathematical Society,
  Providence, R.I., 1967.

\bibitem{Denton2010}
Tom Denton, Florent Hivert, Anne Schilling, and Nicolas~M. Thi{\'e}ry.
\newblock On the representation theory of finite j-trivial monoids.
\newblock {\em S\'em. Lothar. Combin.}, 64:Art. B64d, 44, 2010/11.

\bibitem{Ganyushkin2009b}
Olexandr Ganyushkin and Volodymyr Mazorchuk.
\newblock {\em Classical finite transformation semigroups}, volume~9 of {\em
  Algebra and Applications}.
\newblock Springer-Verlag London Ltd., London, 2009.
\newblock An introduction.

\bibitem{Ganyushkin2009a}
Olexandr Ganyushkin, Volodymyr Mazorchuk, and Benjamin Steinberg.
\newblock On the irreducible representations of a finite semigroup.
\newblock {\em Proc. Amer. Math. Soc.}, 137(11):3585--3592, 2009.

\bibitem{Howie1995}
John~M. Howie.
\newblock {\em Fundamentals of semigroup theory}, volume~12 of {\em London
  Mathematical Society Monographs. New Series}.
\newblock The Clarendon Press Oxford University Press, New York, 1995.
\newblock Oxford Science Publications.

\bibitem{James1981}
Gordon James and Adalbert Kerber.
\newblock {\em The representation theory of the symmetric group}, volume~16 of
  {\em Encyclopedia of Mathematics and its Applications}.
\newblock Addison-Wesley Publishing Co., Reading, Mass., 1981.
\newblock With a foreword by P. M. Cohn, With an introduction by Gilbert de B.
  Robinson.

\bibitem{Li2011}
Liping Li.
\newblock A characterization of finite {EI} categories with hereditary category
  algebras.
\newblock {\em J. Algebra}, 345:213--241, 2011.

\bibitem{Margolis2011}
Stuart Margolis and Benjamin Steinberg.
\newblock The quiver of an algebra associated to the {M}antaci-{R}eutenauer
  descent algebra and the homology of regular semigroups.
\newblock {\em Algebr. Represent. Theory}, 14(1):131--159, 2011.

\bibitem{Margolis2012}
Stuart Margolis and Benjamin Steinberg.
\newblock Quivers of monoids with basic algebras.
\newblock {\em Compos. Math.}, 148(5):1516--1560, 2012.

\bibitem{Pin1986}
J.-E. Pin.
\newblock {\em Varieties of formal languages}.
\newblock Foundations of Computer Science. Plenum Publishing Corp., New York,
  1986.
\newblock With a preface by M.-P. Sch{\"u}tzenberger, Translated from the
  French by A. Howie.

\bibitem{Ponizovski1987}
I.~S. Ponizovski{\u\i}.
\newblock Some examples of semigroup algebras of finite representation type.
\newblock {\em Zap. Nauchn. Sem. Leningrad. Otdel. Mat. Inst. Steklov. (LOMI)},
  160(Anal. Teor. Chisel i Teor. Funktsii. 8):229--238, 302, 1987.

\bibitem{Putcha1996}
Mohan~S. Putcha.
\newblock Complex representations of finite monoids.
\newblock {\em Proc. London Math. Soc. (3)}, 73(3):623--641, 1996.

\bibitem{Putcha1998}
Mohan~S. Putcha.
\newblock Complex representations of finite monoids. {II}. {H}ighest weight
  categories and quivers.
\newblock {\em J. Algebra}, 205(1):53--76, 1998.

\bibitem{Ringel2000}
C.~M. Ringel.
\newblock The representation type of the full transformation semigroup {$T\sb
  4$}.
\newblock {\em Semigroup Forum}, 61(3):429--434, 2000.

\bibitem{Sagan2001}
Bruce~E. Sagan.
\newblock {\em The symmetric group}, volume 203 of {\em Graduate Texts in
  Mathematics}.
\newblock Springer-Verlag, New York, second edition, 2001.
\newblock Representations, combinatorial algorithms, and symmetric functions.

\bibitem{Saliola2007}
Franco~V. Saliola.
\newblock The quiver of the semigroup algebra of a left regular band.
\newblock {\em Internat. J. Algebra Comput.}, 17(8):1593--1610, 2007.

\bibitem{Sam2014}
Steven~V Sam and Andrew Snowden.
\newblock Gr\"obner methods for representations of combinatorial categories.
\newblock {\em arXiv preprint arXiv:1409.1670}, 2014.

\bibitem{Stanley1997}
Richard~P. Stanley.
\newblock {\em Enumerative combinatorics. {V}ol. 1}, volume~49 of {\em
  Cambridge Studies in Advanced Mathematics}.
\newblock Cambridge University Press, Cambridge, 1997.
\newblock With a foreword by Gian-Carlo Rota, Corrected reprint of the 1986
  original.

\bibitem{Steinberg2006}
Benjamin Steinberg.
\newblock M\"obius functions and semigroup representation theory.
\newblock {\em J. Combin. Theory Ser. A}, 113(5):866--881, 2006.

\bibitem{Steinberg2008}
Benjamin Steinberg.
\newblock M\"obius functions and semigroup representation theory. {II}.
  {C}haracter formulas and multiplicities.
\newblock {\em Adv. Math.}, 217(4):1521--1557, 2008.

\bibitem{Steinberg2015}
Benjamin Steinberg.
\newblock The global dimension of the full transformation monoid.
\newblock {\em arXiv preprint arXiv:1502.00959}, 2015.

\bibitem{Webb2007}
Peter Webb.
\newblock An introduction to the representations and cohomology of categories.
\newblock In {\em Group representation theory}, pages 149--173. EPFL Press,
  Lausanne, 2007.

\end{thebibliography}

\end{document}